\documentclass[abstract=no]{scrartcl} 
\usepackage[margin=3.8cm]{geometry}
\usepackage[T1]{fontenc}
\usepackage[utf8]{inputenc}

\usepackage{amsmath}
\usepackage{amssymb}
\usepackage{amsthm}
\usepackage[backend=bibtex,hyperref=true,backref=true,style=alphabetic,sorting=nyt,maxnames=99,minalphanames=4,maxalphanames=4]{biblatex}
\usepackage{bm}
\usepackage{booktabs}
\usepackage{caption}
\usepackage{csquotes}
\usepackage{enumitem}
\usepackage{float}

\usepackage{hyphenat}
\usepackage{mathrsfs}
\usepackage{mathtools}
\usepackage{nth}
\usepackage{url}
\usepackage[final]{microtype}

\usepackage{tikz}
\usetikzlibrary{arrows,matrix,positioning,calc}
\tikzstyle{dmatrix}=[matrix of math nodes,row sep=2.5em, column sep=2.5em,
text height=1.5ex, text depth=0.25ex] 
\tikzstyle{smatrix}=[matrix of math nodes,row sep=1.5em, column sep=1.5em,
text height=1.5ex, text depth=0.25ex] 

\usepackage{bm} 
\usepackage{comment}
\usepackage{mathtools}
\usepackage{tgpagella}
\usepackage{tgheros}
\usepackage{eulervm}
\usepackage{tikz}
\usetikzlibrary{arrows,matrix,positioning}
\tikzstyle{dmatrix}=[matrix of math nodes,row sep=2.5em, column sep=2.5em,
text height=1.5ex, text depth=0.25ex] 

\numberwithin{equation}{section}

\AtBeginEnvironment{abstract}{\small} 
\setparsizes{1em}{0.15\baselineskip plus .25\baselineskip}{1em plus 1fil}

\excludecomment{dummyversion}

\DeclareMathOperator{\Bl}{Bl}     

\DeclareMathOperator{\Pic}{Pic}

\DeclareMathOperator{\Sym}{Sym}

\DeclareMathOperator{\rk}{rk}

\newcommand{\RelSpec}{\mathop{\mathcal{S}{\mathit{pec}}}\nolimits}
\newcommand{\CC} {\mathbb{C}}

\newcommand{\ZZ} {\mathbb{Z}}

\newcommand{\PP}{\mathbb{P}}
\renewcommand{\P} {\mathbb{P}}

\newcommand{\normaliz}[1]{\widetilde{#1}}

\newcommand{\DegenMorph}{\phi}

\newtheorem{proposition}{Proposition}[section]
\newtheorem{theorem}[proposition]{Theorem}
\newtheorem*{theorem*}{Theorem}
\newtheorem{corollary}[proposition]{Corollary}
\newtheorem*{corollary*}{Corollary}

\newtheorem*{conjecture*}{Conjecture}
\newtheorem{lemma}[proposition]{Lemma}
\newtheorem*{lemma*}{Lemma}

\theoremstyle{definition}
\newtheorem{definition}[proposition]{Definition}
\newtheorem*{definition*}{Definition}

\newtheorem*{example*}{Example}
\newtheorem{remark}[proposition]{Remark}
\newtheorem*{remark*}{Remark}

\newtheorem*{question*}{Question}

\newtheorem*{result*}{Result}

\newcommand{\MgSymbol}{\mathcal{M}}
\newcommand{\MgStackSymbol}{\mathsf{M}}
\newcommand{\RgSymbol}{\mathcal{R}}
\newcommand{\RgStackSymbol}{\mathsf{R}}
\newcommand{\AgSymbol}{\mathcal{A}}
\newcommand{\UnivGrdSymbol}{\mathfrak{G}}
\newcommand{\UnivCrdSymbol}{\mathfrak{C}}

\newcommand{\Mg}[1]{\MgSymbol_{#1}}
\newcommand{\Ag}[1]{\AgSymbol_{#1}}
\newcommand{\Rg}[1]{\RgSymbol_{#1}}
\newcommand{\MgStack}[1]{\MgStackSymbol_{#1}}
\newcommand{\RgStack}[1]{\RgStackSymbol_{#1}}
\newcommand{\MgBar}[1]{\overline{\MgSymbol}_{#1}}

\newcommand{\RgBar}[1]{\overline{\RgSymbol}_{#1}}

\newcommand{\RgBarStack}[1]{\overline{\RgStackSymbol}_{#1}}
\newcommand{\BoundaryPCIrr}{\widetilde{\Delta}_0}
\newcommand{\BoundaryPCZero}{\Delta^0_0}

\newcommand{\MgBarPCIrr}[1]{\widetilde{\MgSymbol}_{#1}}
\newcommand{\MgBarStackPCIrr}[1]{\widetilde{\MgStackSymbol}_{#1}}
\newcommand{\MgPCZero}[1]{\MgSymbol^0_{#1}}
\newcommand{\MgStackPCZero}[1]{\MgStackSymbol^0_{#1}}

\newcommand{\MgBarStackPCZero}[1]{\overline{\MgStackSymbol}^0_{#1}}

\newcommand{\RgBarPCIrr}[1]{\widetilde{\RgSymbol}_{#1}}
\newcommand{\RgBarStackPCIrr}[1]{\widetilde{\RgStackSymbol}_{#1}}
\newcommand{\RgBarPCZero}[1]{\overline{\RgSymbol}^0_{#1}}
\newcommand{\RgBarStackPCZero}[1]{\overline{\RgStackSymbol}^0_{#1}}

\newcommand{\RgBarDesing}[1]{\widehat{\RgSymbol}_{#1}}
\newcommand{\RgBarReg}[1]{\RgBar{#1}^{\mathrm{reg}}}
\newcommand{\UnivPrymCurveNaive}{\mathcal{Z}}
\newcommand{\UnivPrymCurve}[1]{\mathcal{X}_{#1}}

\newcommand{\Dram}{\Delta_0^{\mathrm{ram}}}
\newcommand{\dram}{\delta_0^{\mathrm{ram}}}
\newcommand{\dprimeboundary}{( \delta_0' + \delta_0'' )}

\newcommand{\UnivGrd}[2]{\UnivGrdSymbol^{#1}_{#2}}
\newcommand{\UnivGrdPCZero}[2]{\overline{\UnivGrdSymbol}^{#1}_{#2}}
\newcommand{\UnivGrdTors}[2]{\UnivGrdSymbol^{#1,(2)}_{#2}}
\newcommand{\UnivGrdTorsPCZero}[2]{\overline{\UnivGrdSymbol}^{#1,(2)}_{#2}}

\newcommand{\UnivCrdPCZero}[2]{\overline{\UnivCrdSymbol}^{#1}_{#2}}

\newcommand{\UnivCrdTorsPCZero}[2]{\overline{\UnivCrdSymbol}^{#1,(2)}_{#2}}

\newcommand{\GrdSuit}[2]{\UnivGrdTorsPCZero{#1}{#2}}
\newcommand{\CrdSuit}[2]{\UnivCrdTorsPCZero{#1}{#2}}
\newcommand{\RgSuit}[1]{\RgBarPCZero{#1}}

\newcommand{\PrymMap}[1]{\mathcal{P}_{#1}}
\newcommand{\PrymPoincare}{\mathscr{P}}
\newcommand{\concreteF}{\UnivCrdTorsMorph_\ast ( \Poincarebundle^{\otimes 2} )}
\newcommand{\degenmorph}{\phi} 
\newcommand{\Poincarebundle}{\mathscr{L}}
\newcommand{\RgForgetfulMap}{\pi}
\newcommand{\RgStackForgetfulMap}{\pi} 
\newcommand{\UnivPrymCurveMorph}{f}
\newcommand{\UnivCrdMorph}[2]{f^{#1}_{#2}}
\newcommand{\UnivCrdTorsMorph}{\chi}
\newcommand{\UnivGrdTorsMorph}{\sigma}

\newcommand{\Dg}[1]{\mathcal{D}_{#1}}
\newcommand{\Dgbar}[1]{\overline{\mathcal{D}}_{#1}}
\newcommand{\MyDivisor}{\Dg{15}}
\newcommand{\MyDivisorBar}{\Dgbar{15}}

\newcommand{\Addresses}{{
  \bigskip
  \footnotesize
  \textsc{Humboldt-Universität zu Berlin, Institut für Mathematik, Unter den Linden 6}\par
  10099 \textsc{Berlin, Germany}\par\nopagebreak
  \textit{E-mail address}:  \texttt{gregor.bruns@hu-berlin.de}

}}

\title{$\bm{\RgBar{15}}$ is of general type}
\author{Gregor Bruns}
\date{}

\begin{document}

\maketitle

\begin{abstract}
  We prove that the moduli space $\RgBar{15}$ of Prym curves of genus $15$ is of general type.
  To this end we exhibit a virtual divisor $\MyDivisorBar$ on $\RgBar{15}$ as the degeneracy locus of a
  globalized multiplication map of sections of line bundles.  We then proceed to show that this locus is
  indeed of codimension one and calculate its class.  Using this class, we can conclude that $K_{\RgBar{15}}$ is big.
  This complements a 2010 result of Farkas and Ludwig:  now the spaces $\RgBar{g}$ are known to be of
  general type for $g \geq 14$.
\end{abstract}

\section{Introduction}

The study of Prym varieties has a long history, going back to work of Riemann, Wir\-tin\-ger, Schottky and Jung
in the late \nth{19} and early \nth{20} century.  Of particular interest is the correspondence between moduli
of étale double covers of curves of genus $g$ and abelian varieties of dimension $g-1$, given by the Prym map
$\PrymMap{g}\colon \Rg{g} \rightarrow \Ag{g-1}$.
Here we denote by $\Rg{g}$ the moduli space of pairs $[C,\eta]$ where $[C]\in \Mg{g}$ is a
smooth genus $g$ curve and $\eta \in \Pic^0(C)$ is a $2$-torsion point (or equivalently
an étale double cover of $C$).

It turns out that every principally polarized abelian variety (ppav) up to dimension $5$ is a Prym variety.
This generalizes the well-known fact that the general ppav of dimension at most $3$ is the Jacobian of a curve.
In dimension greater than $5$, Prym varieties are no longer dense in the moduli space of ppavs,
but their locus is still of geometric interest.

It is natural to ask for a modular compactification of $\Rg{g}$ in order to study
degenerations of Prym varieties and the birational geometry of their families.
Explicit constructions were put forward in \autocite{Beau77,Ber99} and in \autocite{BCF}
where the compactification is given in terms of admissible covers and Prym curves, respectively.

Much is already known  about the birational geometry of $\Rg{g}$.
It is a rational variety for $g \leq 4$, unirational for $g \leq 7$ and uniruled for $g \leq 8$
(see \autocite{FV2016} for a discussion).  The availability of a modular compactification
has sparked interest in the Kodaira dimension of $\RgBar{g}$ for higher $g$.
Farkas and Ludwig \autocite{FL} prove that $\RgBar{g}$ is of general type
for $g \geq 14$ and $g \not= 15$.  The Kodaira dimension of $\RgBar{12}$ is shown to be nonnegative.

In this note we close the gap at $g = 15$.
\begin{theorem}
  The moduli space $\RgBar{15}$ is of general type.
\end{theorem}
We briefly outline the strategy of the proof.
In order to show that the canonical class of $\RgBar{15}$
is big, we construct an effective divisor $\MyDivisor$ such that $K_{\RgBar{15}}$
can be written as a positive linear combination of the Hodge class, the class of $\MyDivisorBar$
and other effective divisor classes.

To motivate the construction of $\MyDivisor$, consider first the case of genus $6$.
A general curve $[C] \in \Mg{6}$ possesses a finite number of complete $\mathfrak{g}^2_6$.
Any such $L \in W^2_6(C)$ induces a birational map to a plane sextic curve $\Gamma$ with $4$ nodes.
If there is a plane conic $Q$ totally tangent to $\Gamma$, i.e.,
$Q \cdot \Gamma = 2D$ where $D$ is effective of degree $6$, then
$\eta = \mathcal{O}_\Gamma(-1) \otimes \mathcal{O}_\Gamma(D)$ is $2$-torsion.
\begin{center}
  \begin{tikzpicture}
    \node[inner sep=0] at (0,0) { \includegraphics[scale=0.5]{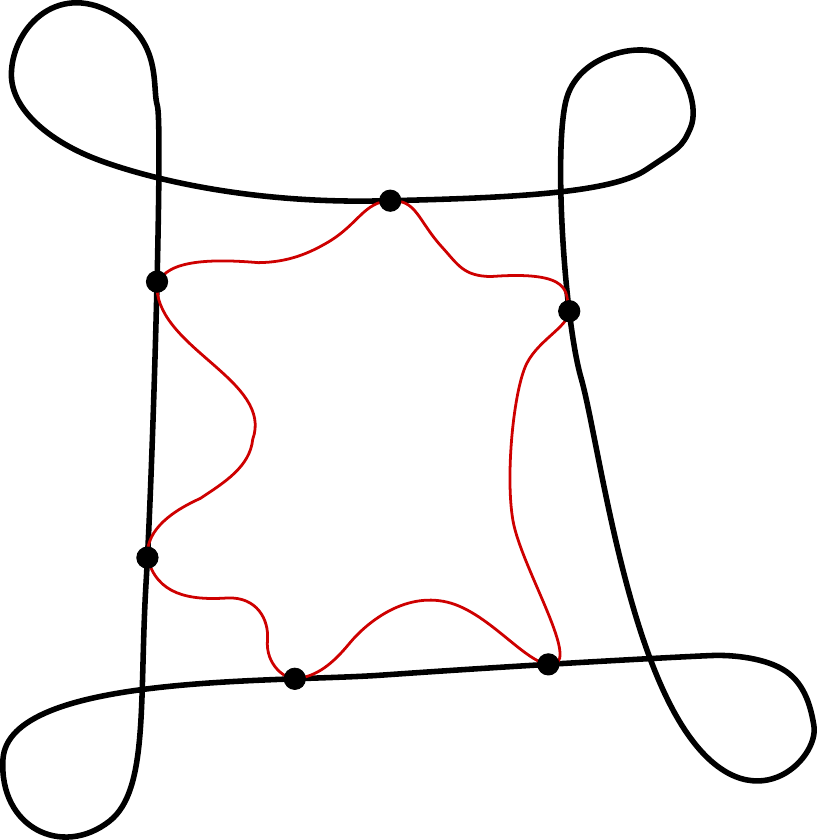} };
    \node at (-0.55,-0.1) { $Q$ };
    \node at (1.7,1.4) { $\Gamma$ };
  \end{tikzpicture}
\end{center}
The existence of such a totally tangent conic is equivalent to the failure of the map
\begin{equation*}
  \Sym^2 H^0(C, L \otimes \eta) \rightarrow \frac{H^0(C,L^{\otimes 2})}{\Sym^2 H^0(C,L)}
\end{equation*}
to be injective.  It turns out that the closure of the locus of pairs $[C,\eta] \in \Rg{6}$
where this injectivity fails is a divisor, i.e., the condition to possess a totally tangent conic
to a plane sextic model gives a divisorial condition on $\Rg{6}$.  This divisor can also
be identified with the closure of the ramification divisor of the Prym map $\Rg{6} \rightarrow \Ag{5}$.
For details, see \autocite{FGSV}.

We generalize this condition and adapt it to genus $15$.  A general genus $15$ curve $C$
carries a finite number of complete $\mathfrak{g}^4_{16}$ linear series.
For any such $L \in W^4_{16}(C)$ we can consider the multiplication map
\begin{equation*}
  \mu_{[C,L]} \colon \Sym^2 H^0(C,L) \rightarrow H^0(C, L^{\otimes 2}).
\end{equation*}
The vector spaces on the left- and right-hand sides are of dimensions $15$ and $18$, respectively,
and the map is injective for the general pair $[C,L]$.  We can make use of a torsion bundle $\eta$
to get the remaining three sections:
\begin{equation}
  \label{eq:multiplication-map-intro}
  \mu_{[C,\eta,L]}\colon \Sym^2 H^0(C, L) \oplus \Sym^2 H^0(C, L\otimes \eta) \rightarrow H^0(C, L^{\otimes 2}).
\end{equation}
We consider the locus of Prym curves carrying a $\mathfrak{g}^4_{16}$ such that this map
fails to be an isomorphism.  Unlike in genus $6$, such curves are not directly characterized
by having a totally tangent quadric hypersurface, although on those that have, the map
\eqref{eq:multiplication-map-intro} certainly fails to be injective.

It turns out that $\mu_{[C,\eta,L]}$ is bijective
for all $L$ on the general pair $[C,\eta]\in \Rg{15}$ and the failure locus is in codimension one.
We may therefore consider the divisor
\begin{equation*}
  \MyDivisor = \left\{ [C, \eta] \in \Rg{15} ~\left|~
  \exists L\in W^4_{16}(C) \text{ such that } \mu_{[C,\eta,L]}\text{ is not an isomorphism} \right.\right\}.
\end{equation*}
In order to show that \eqref{eq:multiplication-map-intro} is indeed bijective for all $\eta$ and $L$
on a general curve $C$, we first construct in Section \ref{subsec:curve-example} a single example, using a curve
that carries a theta characteristic with a large number of sections.  Afterwards we
prove that the moduli space $\UnivGrdTors{4}{16}$ of triples $[C,\eta,L]$ is irreducible,
allowing us to specialize the general triple to the constructed example.  More generally, we
obtain the following result:
\begin{proposition}
  Assume $g\geq 3$ and the Brill--Noether number $\rho(g,r,d) = 0$.  If either $r \leq 2$ 
  or $g - d + r - 1 \leq 2$ then $\UnivGrdTors{r}{d}$ is irreducible.
\end{proposition}
Taking the closure $\MyDivisorBar$ of $\MyDivisor$ in an appropriate partial compactification
$\RgSuit{15}$ of $\Rg{15}$,
we can calculate the class of $\MyDivisorBar$ using a determinantal description
coming from globalizing the map \eqref{eq:multiplication-map-intro} to a morphism
of vector bundles.
\begin{theorem}
  The locus $\MyDivisorBar$ is a divisor in $\RgSuit{15}$ and we have the expression
  \begin{equation*}
    [\MyDivisorBar] + E \equiv
    31020 \left( \frac{3127}{470} \lambda - (\delta_0' + \delta_0'') - \frac{3487}{1880} \dram \right)
  \end{equation*}
  where $E$ is an effective class on $\RgSuit{15}$.
\end{theorem}
A suitable positive linear combination of $\MyDivisorBar$ and another divisor $\overline{\mathcal{D}}_{15:2}$,
which was described in \autocite{FL}, then shows that the canonical class of $\RgBar{15}$ is big.

To be able to calculate the class of $\MyDivisorBar$, various technical difficulties have
to be overcome.  In Section \ref{sec:rg-construction}
we closely follow the set-up of \autocite{F09,FL}
to construct partial compactifications of suitable open subsets of $\Mg{g}$ and $\Rg{g}$
and to extend the moduli stacks of linear series there.  We also make use of a result in \autocite{FL}
showing that all pluricanonical forms defined on the smooth part of $\RgBar{g}$ extend
to any resolution of singularities.

\paragraph{Acknowledgements}
This project was generously supported by the
Gra\-du\-ier\-tenkolleg 1800 of the Deutsche Forschungsgemeinschaft
and by the Berlin Mathematical School.
I would like to thank my PhD advisor Gavril Farkas
for his suggestion to study the circle of problems to which this question belongs, as well
as for many helpful discussions.  
I also thank the anonymous referee for careful reading and helpful comments.

\section{The moduli space of Prym curves}
\label{sec:rg-construction}

We follow the techniques and notations introduced in \autocite[Section 1]{FL}.
First we recall the basic definitions.

A \emph{smooth Prym curve} is a pair $[C,\eta]$ where $[C]\in \Mg{g}$ is a smooth
curve and $\eta \in \Pic^0(C)\setminus\{\mathcal{O}_C\}$ is such that $\eta^{\otimes 2} \cong \mathcal{O}_C$.
To such a pair we can naturally associate an étale double cover $f\colon C' \rightarrow C$
where $C'$ is given as $\RelSpec( \mathcal{O}_C \oplus \eta )$.  Conversely, every étale
double cover determines a unique $2$-torsion bundle $\eta$ on $C$.

We denote by $\Rg{g}$ the moduli
space of smooth Prym curves of genus $g$ and by $\RgForgetfulMap\colon \Rg{g} \rightarrow \Mg{g}$
the forgetful morphism $[C,\eta] \mapsto [C]$ of degree $2^{2g} - 1$.
The corresponding morphism on stacks is étale and denoted by $\RgForgetfulMap\colon \RgStack{g} \rightarrow \MgStack{g}$ as well.

\subsection{Compactifying the space of Prym curves}

In order to compactify $\Rg{g}$, we make the following definitions.
We say that a smooth rational component of a nodal curve is \emph{exceptional}
if it meets the other components in exactly two points.
A nodal curve is called \emph{quasistable} if every smooth rational
component meets the rest of the curve in at least two points,
and the exceptional components are pairwise disjoint.
\begin{definition}
  A \emph{Prym curve} of genus $g$ is a triple $(C, \eta, \beta)$ consisting of a quasistable
  curve $C$ of genus $g$,
  a line bundle $\eta\in\Pic^0(C)$
  and a sheaf homomorphism $\beta\colon \eta^{\otimes 2} \rightarrow \mathcal{O}_C$,
  subject to the following conditions:
  \begin{enumerate}
  \item For each exceptional component $E$ of $C$ we have $\eta|_E = \mathcal{O}_E(1)$.
  \item For each nonexceptional component the morphism $\beta$ is not the zero morphism.
  \end{enumerate}
  A \emph{family of Prym curves} over a scheme $S$ is a triple $(\mathcal{C} \rightarrow S, \eta, \beta)$
  where $\mathcal{C} \rightarrow S$ is a a flat family of quasistable curves, $\eta$
  is a line bundle on $\mathcal{C}$ and $\beta\colon \eta^{\otimes 2} \rightarrow \mathcal{O}_{\mathcal{C}}$
  is a sheaf homomorphism such that for each fiber $C_s = \mathcal{C}(s)$ the triple $(C_s, \eta|_{C_s}, \beta|_{C_s})$
  is a Prym curve.
\end{definition}
If there is no danger of confusion, we usually omit the morphism $\beta$
from the data to describe a Prym curve.
We denote by $\RgBarStack{g}$ the (nonsingular Deligne--Mumford) stack of Prym curves of genus $g$
and its coarsening by $\RgBar{g}$.  The locus $\Rg{g}$ of smooth Prym curves is
contained in $\RgBar{g}$ as an open subset and the forgetful map $\RgForgetfulMap$ extends
to a ramified covering $\RgBar{g} \rightarrow \MgBar{g}$ which we also denote by $\RgStackForgetfulMap$.
Note that by blowing down all exceptional components of a quasistable curve we obtain a stable curve.
It should also be remarked that there is a close relationship between the Prym curves discussed here
and admissible covers in the sense of Beauville \autocite{Beau77}.
For a detailed treatment of the previous statements, see \autocite{BCF,Ber99}.

\subsection{Boundary divisors}

We study the boundary components of $\RgBar{g}$.
They lie over the boundary of $\MgBar{g}$, so we can study
the components lying over $\Delta_i$ for $i = 0,\dotsc,\lfloor \frac{g}{2} \rfloor$.
As is customary, we denote by $\delta_i$ the corresponding divisor classes.

\paragraph{The divisors $\bm{{\Delta_i}},\bm{{\Delta}_{g-i}},\bm{{\Delta}_{g:i}}$ for $\bm{{i\geq 1}}$.}

First consider $i\geq 1$
and let $X \in \Delta_i$ be general, i.e., $X = C\cup D$ is the union of two curves of genera $i$ and $g-i$
meeting transversally in a single node.  The line bundle $\eta\in\Pic^0(X)$ on the corresponding Prym curve
is determined by its restrictions $\eta_C = \eta|_C$ and $\eta_D = \eta|_D$ satisfying
$\eta_C^{\otimes 2} = \mathcal{O}_C$ and $\eta_D^{\otimes 2} = \mathcal{O}_D$.

Either one of $\eta_C$ and $\eta_D$ (but not both) can be trivial, so $\RgForgetfulMap^\ast(\Delta_i)$ splits into three
irreducible components
\begin{equation*}
  \RgForgetfulMap^\ast(\Delta_i) = \Delta_i + \Delta_{g-i} + \Delta_{i:g-i}
\end{equation*}
where the general element in $\Delta_i$ is $[C\cup D, \eta_C \not= \mathcal{O}_C, \mathcal{O}_D]$,
the generic point of $\Delta_{g-i}$ is of the form $[C\cup D, \mathcal{O}_C, \eta_D \not= \mathcal{O}_D]$
and the generic point of $\Delta_{i:g-i}$ looks like $[C\cup D, \eta_C \not= \mathcal{O}_C, \eta_D\not= \mathcal{O}_D]$.

\paragraph{The divisor $\bm{{\Delta_0''}}$.}

Now let $i = 0$.  The generic point of $\Delta_0$ in $\MgBar{g}$ is a one-nodal irreducible curve $C$
of geometric genus $g - 1$.  We first consider points of the form $[C, \eta]$ lying over $C$, i.e.,
without an exceptional component.  Denote by $\nu\colon \normaliz{C} \rightarrow C$ the normalization
and by $p,q$ the preimages of the node.
Then we have an exact sequence
\begin{equation*}
  0 \rightarrow \CC^\ast \rightarrow \Pic^0(C) \xrightarrow{\nu^\ast} \Pic^0(\normaliz{C}) \rightarrow 0 
\end{equation*}
which restricts to
\begin{equation*}
  0 \rightarrow \ZZ/2\ZZ \rightarrow \Pic^0(C)[2] \xrightarrow{\nu^\ast} \Pic^0(\normaliz{C})[2] \rightarrow 0 
\end{equation*}
on the $2$-torsion part.  The group $\ZZ/2\ZZ$ represents the two possible choices of gluing of the fibers
at $p$ and $q$ for each line bundle in $\Pic^0(\normaliz{C})[2]$.
For the case $\nu^\ast \eta = \mathcal{O}_{\normaliz{C}}$ there is exactly
one possible choice of $\eta \not= \mathcal{O}_C$.  These curves $[C, \eta]$ correspond to the classical
\emph{Wirtinger covers}
\begin{equation*}
  \normaliz{C}_1\amalg \normaliz{C}_2 / (p_1 \sim q_2, p_2 \sim q_1)  \xrightarrow{2:1}  \normaliz{C}/(p\sim q) = C.
\end{equation*}
We denote by $\Delta_0''$ the closure of the locus of Wirtinger covers.

\paragraph{The divisor $\bm{{\Delta_0'}}$.}

On the other hand, there are $2^{2(g-1)} - 1$ nontrivial elements in the group $\Pic^0(\normaliz{C})[2]$.
For each of them there are two choices of gluing, so we have a total of $2 \cdot (2^{2g -2} - 1)$ choices for
$\eta$ such that $\nu^\ast \eta \not= \mathcal{O}_{\normaliz{C}}$.  We let $\Delta_0'$ be the closure of the locus
of pairs $[C,\eta]$ such that $\nu^\ast\eta \not = \mathcal{O}_{\normaliz{C}}$.

\paragraph{The divisor $\bm{{\Dram}}$.}

Let us turn to the case of curves of the form $[X = \widetilde{C}\cup_{p,q} E, \eta]$ where $E$ is an exceptional component.
The stabilization of such a curve is again a one-nodal curve $C$.
Denote by $\beta$ the morphism $\eta^{\otimes 2} \rightarrow \mathcal{O}_X$.
Since $\eta|_E = \mathcal{O}_E(1)$, we must have $\beta_{E\setminus\{p,q\}} = 0$
and $\deg(\eta^{\otimes 2}|_{\normaliz{C}}) = -2$.
It follows that $\eta^{\otimes 2}|_{\normaliz{C}} = \mathcal{O}_{\normaliz{C}} ( - p - q )$.
There are $2^{2(g-1)}$ choices of square roots
of $\mathcal{O}_{\normaliz{C}}(-p-q)$ and each of these determines uniquely
a Prym curve $[X, \eta]$ of this form.
We denote the closure of the locus of such curves by $\Dram$.

As a result of the local analysis carried out for instance in \autocite{CEFS}, we see that $\RgForgetfulMap$ is simply
ramified over $\Dram$ and unramified everywhere else.  This gives the relation
\begin{equation*}
  \RgForgetfulMap^\ast (\delta_0) = \delta_0' + \delta_0'' + 2\dram.
\end{equation*}
\begin{dummyversion}
  \begin{remark}
    We also have the following relations on pushforwards, a result of the analysis carried out in the previous Section:
    \begin{equation*}
      \RgForgetfulMap_\ast  (\delta_0') = 2 \cdot (2^{2g-2} - 1)\delta_0,\quad \RgForgetfulMap_\ast (\delta_0'') = \delta_0,\quad \RgForgetfulMap_\ast (\dram) = 2^{2g-2}\delta_0
    \end{equation*}
  \end{remark}
\end{dummyversion}

\subsection{The canonical class}

In order to show that $\RgBar{g}$ is of general type, we need to show the canonical class is
big for some desingularization $\RgBarDesing{g}$ of $\RgBar{g}$.
Using the following extension result we see that all pluricanonical differentials on the smooth part of
$\RgBar{g}$ extend to $\RgBarDesing{g}$.
\begin{theorem}[{\cite[Theorem 6.1]{FL}}]
  Let $g\geq 4$ and $\RgBarDesing{g} \rightarrow \RgBar{g}$ be any desingularization.
  Then every pluricanonical form defined on the smooth locus $\RgBarReg{g}$ of $\RgBar{g}$
  extends holomorphically to $\RgBarDesing{g}$; that is,
  for all integers $l\geq 0$ we have isomorphisms
  \begin{equation*}
    H^0 \Big( \RgBarReg{g}, K^{\otimes l}_{\RgBar{g}} \Big)
    \cong
    H^0 \Big( \RgBarDesing{g}, K^{\otimes l}_{\RgBarDesing{g}} \Big).
  \end{equation*}
\end{theorem}
Furthermore, one has the expression
\begin{equation*}
  K_{\RgBar{g}} = 13 \lambda - 2 \dprimeboundary - 3 \dram
  - 2 \sum_{i = 1}^{\lfloor g/2 \rfloor} ( \delta_i + \delta_{g-i} + \delta_{i:g-i} )
  - (\delta_1 + \delta_{g-1} + \delta_{1:g-1})
\end{equation*}
for the canonical class $K_{\RgBar{g}}$
in terms of the divisor classes introduced before (see for example \autocite[Theorem 1.5]{FL}).
Here we have abused notation and set $\lambda = \RgForgetfulMap^\ast(\lambda)$,
the pullback of the Hodge class from $\MgBar{g}$.  
It is therefore enough to exhibit an effective divisor $D$ of the form
\begin{equation*}
  D = a \lambda - (b_0' \delta_0' + b_0'' \delta_0'') - b_0^{\mathrm{ram}} \dram
  - \sum_{i = 1}^{\lfloor g/2 \rfloor} (b_i \delta_i + b_{g-i} \delta_{g-i} + b_{i:g-i} \delta_{i:g-i} )
\end{equation*}
such that
\begin{equation*}
  \frac{a}{\gamma} < \frac{13}{2}\quad\text{ for all } \gamma \in
  \left\{b_0', b_0''\right\} \cup \left.\left\{b_i, b_{g-i}, b_{i:g-i} ~\right|~ i = 1,\dotsc,\lfloor g/2 \rfloor \right\}
\end{equation*}
as well as
\begin{equation*}
  \frac{a}{\gamma} < \frac{13}{3}\quad \text{ for all } \gamma \in \left\{b_0^{\mathrm{ram}}, b_1, b_{g-1}, b_{1:g-1} \right\}.
\end{equation*}
\begin{remark}
  \label{rem:coefficients-bounded}
  Actually, the situation turns out to be simpler.
  Proposition 1.9 of \autocite{FL} shows that for $g \leq 23$ it is enough
  to consider the coefficients of $\lambda$,
  $\delta_0'$, $\delta_0''$ and $\dram$.  If they satisfy the inequalities above,
  the other boundary divisor coefficients are automatically suitably bounded.
  We will make full use of the fact that we do not have to consider singular curves of compact type.
\end{remark}

\subsection{The universal Prym curve}

Since we are only concerned with the boundary divisors $\Delta_0'$, $\Delta_0''$ and $\Dram$,
we partially compactify $\Mg{g}$ by adding the open sublocus $\BoundaryPCIrr \subset \Delta_0$
of one-nodal irreducible curves.  Set
\begin{equation*}
  \MgBarPCIrr{g} = \Mg{g} \cup \BoundaryPCIrr
\end{equation*}
and let $\RgBarPCIrr{g} = \RgForgetfulMap^{-1}( \MgBarPCIrr{g} )$.  We also set
\begin{equation*}
  \UnivPrymCurveNaive = \RgBarStackPCIrr{g} \times_{\MgBarStackPCIrr{g}} \MgBarStackPCIrr{g,1}.
\end{equation*}
This is not yet the universal Prym curve over $\RgBarStackPCIrr{g}$, since the points
on exceptional components of curves in $\Dram$ are not present.  We have to blow up $\UnivPrymCurveNaive$
along the locus $V$ of points
\begin{equation*}
  (X, \eta_X, p = q) \in \Dram \times_{\MgBarStackPCIrr{g}} \MgBarStackPCIrr{g,1}, \quad
  X = C \cup_{\{p,q\}} E \rightarrow C/p\sim q,\quad \eta_E = \mathcal{O}_E(1).
\end{equation*}
Set $\UnivPrymCurve{g} = \Bl_V(\UnivPrymCurveNaive)$ and let
$\UnivPrymCurveMorph\colon \UnivPrymCurve{g} \rightarrow \RgBarStackPCIrr{g}$ be the induced universal
family of Prym curves.
The family $\UnivPrymCurve{g}$ comes equipped with a Poincaré bundle $\PrymPoincare$
such that $\PrymPoincare|_{\UnivPrymCurveMorph^{-1} ([X,\eta,\beta])} = \eta$.
We need the following result from \autocite[Proposition 1.6]{FL}:
\begin{lemma}
  \label{lem:pushforward-dram}
  In $\Pic( \RgBarPCIrr{g} )$ we have $\UnivPrymCurveMorph_\ast ( c_1^2 ( \mathscr{P} ) ) = -\dram/2$
  and $\UnivPrymCurveMorph_\ast ( c_1(\PrymPoincare) c_1(\omega_\UnivCrdTorsMorph)) = 0$.
\end{lemma}

\subsection{Moduli spaces of linear series over the Prym moduli space}

To compute the classes of divisors on $\RgBar{g}$, a viable method is to give them a determinantal description,
i.e., exhibit them as degeneracy loci of morphisms of vector bundles.  To obtain these vector bundles,
we consider the stack $\UnivGrdTors{r}{d}$ parametrizing triples
$[C,\eta,L]$ where $[C,\eta] \in \Rg{g}$
and $L \in G^r_d(C)$.  Note that in the case $\rho(g,r,d) = 0$ in which we are interested, the forgetful map
$\UnivGrdTors{r}{d} \rightarrow \Rg{g}$ is a generically finite cover of degree
\begin{equation*}
  N = g!\, \frac{1!\, 2! \dotsb r!}{(g - d + r)! \dotsb (g - d + 2r)!}.
\end{equation*}
We want to first restrict this construction to an open subset of $\Rg{g}$ such that
various pushforwards of the Poincaré bundles on the universal curve are indeed vector
bundles on $\UnivGrdTors{r}{d}$.  Then we shall extend the stack over a suitable partial
compactification to be able to also determine the behavior on the boundary.

Let $\MgStackPCZero{g}$ be the open substack of $\MgStack{g}$ classifying curves $C$
with $W^{r+1}_{d}(C) = \emptyset$ and $W^{r}_{d-1}(C) = \emptyset$.  A general such curve
indeed has a finite amount of $\mathfrak{g}^r_d$ linear series and all of them are very ample.
Observe that both
\begin{equation*}
  \rho(g, r+1, d) = - (g - d + 2(r+1)) \leq -2, \quad \rho(g, r, d-1) = -(r+1) \leq -2,
\end{equation*}
so the codimension of the complement of $\MgPCZero{g}$ in $\Mg{g}$ is at least $2$,
for instance by results in \autocite{EH89}.  Therefore,
restricting to $\MgPCZero{g}$ does not change divisor class calculations.

To partially compactify $\MgStackPCZero{g}$, add the locus $\BoundaryPCZero$ of
Brill--Noether general irreducible one-nodal curves, i.e., $[C/p\sim q]$ with $[C]\in \Mg{g-1}$
satisfying the Brill--Noether theorem.  Denote by $\MgBarStackPCZero{g} = \MgStackPCZero{g} \cup \BoundaryPCZero$
the resulting partial compactification.  Over $\MgBarStackPCZero{g}$ we consider
the stack of pairs $[C,L]$ where $L \in G^r_d(C)$.  We denote this stack by $\UnivGrdPCZero{r}{d}$.
Pulling back the universal curve
$\MgBarStackPCZero{g,1}$ to $\UnivGrdPCZero{r}{d}$, we get a universal family
\begin{equation*}
  \UnivCrdMorph{r}{d}\colon \UnivCrdPCZero{r}{d}
  = \UnivGrdPCZero{r}{d} \times_{\MgBarStackPCZero{g}} \MgBarStackPCZero{g,1}
  \rightarrow \UnivGrdPCZero{r}{d}
\end{equation*}
and we choose a Poincaré bundle, i.e., an $\Poincarebundle \in \Pic(\UnivCrdPCZero{r}{d})$ such that
$\Poincarebundle|_{(\UnivCrdMorph{r}{d})^{-1}([C,L])} = L$ for every $[C,L] \in \UnivGrdPCZero{r}{d}$.

We are now ready to pull these constructions back to Prym curves.
We let $\RgBarStackPCZero{g} = \RgStackForgetfulMap^{-1}( \MgBarStackPCZero{g} )$ and
\begin{equation*}
  \UnivGrdTorsMorph\colon \UnivGrdTorsPCZero{r}{d}
  = \UnivGrdPCZero{r}{d} \times_{\MgBarStackPCZero{g}} \RgBarStackPCZero{g}
  \rightarrow \RgBarStackPCZero{g}
\end{equation*}
be the stack parametrizing triples $[C,\eta,L]$ for $[C,\eta]\in \RgBarStackPCZero{g}$
and $L \in W^r_d(C)$.  We also have the universal curve
\begin{equation*}
  \UnivCrdTorsMorph\colon \UnivCrdTorsPCZero{r}{d}
  = \UnivPrymCurve{g} \times_{\RgBarStackPCZero{g}} \UnivGrdTorsPCZero{r}{d}
  \rightarrow \UnivGrdTorsPCZero{r}{d}.
\end{equation*}
By pulling back from $\RgBarStackPCZero{g}$ and $\UnivGrdTorsPCZero{r}{d}$, respectively, this comes equipped with
two Poincaré bundles $\PrymPoincare$ and $\Poincarebundle$.  We can also use $\UnivGrdTorsMorph$ to
pull back the boundary classes $\Delta_0'$, $\Delta_0''$ and $\Dram$ from $\RgBarStackPCZero{g}$ to
$\UnivGrdTorsPCZero{r}{d}$.  Slightly abusing notation, the pullbacks will be denoted by the same symbols. 

\section{A new divisor on \texorpdfstring{$\bm{\RgBar{15}}$}{R15}}

As before, we denote by $\UnivCrdTorsMorph\colon \CrdSuit{4}{16} \rightarrow \GrdSuit{4}{16}$
the universal curve and let $\Poincarebundle$ be a Poincaré bundle on $\CrdSuit{4}{16}$.
Furthermore, let $\omega_\UnivCrdTorsMorph$ be the relative dualizing sheaf of $\UnivCrdTorsMorph$ and
$\UnivGrdTorsMorph\colon \GrdSuit{4}{16} \rightarrow \RgSuit{15}$
be the generically finite cover of degree $N = 6006$.

By construction of our moduli stacks and Grauert's theorem,
the pushforwards of $\Poincarebundle$ and $\Poincarebundle^{\otimes 2}$
by $\UnivCrdTorsMorph$ are vector bundles on $\GrdSuit{4}{16}$ of ranks $5$ and $18$, respectively.
The sheaf $\UnivCrdTorsMorph_\ast (\Poincarebundle\otimes \PrymPoincare)$ is possibly not a
vector bundle, but at least it is torsion-free.  By excluding the subvariety
(of codimension at least two) where it fails to be locally free we can assume it is in fact
a vector bundle of rank $2$.  Divisor class calculations will not be affected.

We may then consider the following morphism of vector bundles of the same rank:
\begin{equation*}
  \degenmorph\colon \Sym^2 \UnivCrdTorsMorph_\ast (\Poincarebundle) \oplus \Sym^2 \UnivCrdTorsMorph_{\ast}
  (\Poincarebundle \otimes \PrymPoincare) \rightarrow \UnivCrdTorsMorph_\ast (\Poincarebundle^{\otimes 2}).
\end{equation*}
On the fiber over the class of a triple $[C,\eta,L]$ it is given by the multiplication map of sections
\begin{equation}
  \label{eq:multiplication-map}
  \mu_{[C,\eta,L]}\colon \Sym^2 H^0(C, L) \oplus \Sym^2 H^0(C, L\otimes \eta) \rightarrow H^0(C, L^{\otimes 2}).
\end{equation}
The closure of the locus
\begin{equation*}
  \MyDivisor = \left\{ [C, \eta] \in \Rg{15} ~\left|~
  \exists L\in W^4_{16}(C) \text{ such that } \mu_{[C,\eta,L]}\text{ is not an isomorphism} \right.\right\}
\end{equation*}
therefore has a determinantal description as the pushforward of the first degeneracy locus of the map $\degenmorph$.
Its expected codimension is one and we obtain a virtual divisor.
Note that while the vector bundles involved in defining $\degenmorph$ clearly depend on the choice of
the Poincaré bundle $\Poincarebundle$, the resulting morphism $\degenmorph$ does not (cf. the
remark before Theorem 2.1 in \autocite{F09}).

\subsection{Proof of divisoriality of \texorpdfstring{$\bm{{\MyDivisor}}$}{D15}}
\label{subsec:curve-example}

We now prove that $\MyDivisorBar$ is a genuine divisor, that is, $\mu_{[C,\eta,L]}$ is an isomorphism
for every $L\in W^4_{16}(C)$ on the general Prym curve $[C,\eta]$.  We will prove in 
Section \ref{subsec:grd-irreducible} that $\UnivGrdTors{4}{16}$ over the whole space $\Rg{15}$
is irreducible.
Hence it will be enough to exhibit a single smooth curve $C$ and two line bundles
$L \in W^4_{16}(C)$ and $\eta\in\Pic^0(C)[2]$ such that the multiplication map
\eqref{eq:multiplication-map} is bijective.  We can then specialize the general element
of $\UnivGrdTors{4}{16}$ to this particular example and conclude by semicontinuity.

We start with a smooth nonhyperelliptic curve
$C \in \mathcal{M}_{15}$ possessing a theta characteristic $\vartheta$
with an exactly $5$-dimensional space of global sections, i.e., $|\vartheta|\in G^4_{14}(C)$ and
$\vartheta^{\otimes 2} \cong \omega_C$.
In order to construct an $L$ such that $\mu_{[C,\eta,L]}$ is bijective, $C$ should in
fact be half-canonically embedded by $\vartheta$ such that the image does not lie
on any quadric hypersurface in $\P^4$.  

Explicit models of such curves can be obtained as hyperplane sections of smooth canonical surfaces
$S \subseteq \P^5$ with $p_g = 6$ and $K_S^2 = 14$.
To construct such a surface, one can employ the method described by Catanese \autocite{Cat97}.
\begin{lemma}
  There exists a smooth projective surface $S$ of general type with invariants
  $(K_S^2, p_g, q) = (14, 6, 0)$, canonically embedded in $\P^5$,
  not lying on any quadric hypersurface.
\end{lemma}
\begin{proof}
  The surfaces $S$ arise from Pfaffian resolutions
  \begin{equation}
    \label{eq:pfaffian-resolution}
    0 \rightarrow \mathcal{O}_{\P^5} (-7) \rightarrow \mathcal{O}_{\P^5} (-4)^{\oplus 7}
    \xrightarrow{\alpha} \mathcal{O}_{\P^5}(-3)^{\oplus 7} \xrightarrow{p} \mathcal{I}_S \rightarrow 0
  \end{equation}
  of the ideal sheaf $\mathcal{I}_S$,
  where $\alpha$ is a $7\times 7$ antisymmetric matrix with linear entries and $p$ is the map given by the Pfaffians
  of $6\times 6$ principal submatrices of $\alpha$.

  Using the projective resolution \eqref{eq:pfaffian-resolution} and Serre duality for Ext sheaves, we see that
  $S$ is canonically embedded.  We also see that $S$ is a regular surface (i.e., $q = 0$) and $p_g = 6$,
  which combines to give $\chi(\mathcal{O}_S) = 7$.
  Again using \eqref{eq:pfaffian-resolution}, the Hilbert polynomial of $\mathcal{O}_S$ is
  $P_S(t) = 7t^2 - 7t + 7$,
  which tells us $\deg(S) = 14$, and because $S$ is canonically embedded we have $K_S^2 = 14$.
\end{proof}
A general hyperplane section $C = H\cap S$ of $S$ has, by the adjunction formula,
\begin{equation*}
  \omega_C \cong (\mathcal{O}_S(1)\otimes \omega_S)|_C \cong \omega_S^{\otimes 2}|_C,\quad 2g - 2 = 2K_S \cdot K_S = 28,
\end{equation*}
so $C \hookrightarrow \P^4$ is half-canonically embedded of degree $14$ and genus $15$.
Using the exact sequence
\begin{equation*}
  0 \rightarrow \mathcal{I}_S(2) \rightarrow \mathcal{O}_{\P^5}(2) \rightarrow \mathcal{O}_S(2) \rightarrow 0
\end{equation*}
and $h^0(S,\omega_S^{\otimes 2}) = 21$ by Riemann--Roch, we get $H^0(\P^5,\mathcal{I}_S(2)) = 0$,
so $S$ does not lie on a quadric hypersurface of $\P^5$.  The same then applies for $C$ in $\PP^4$.
A moduli count shows that hyperplane sections of such $S$ form a $32$-dimensional family.
\begin{remark}
  This is not the only way in which such curves arise.  Iliev and Markushevich \autocite{IM00} also obtain
  a $32$-dimensional family (i.e., an irreducible component of the expected dimension of the locus
  $\mathcal{T}^4_{15}$ of curves of genus $15$ having a theta-characteristic with $5$ independent
  global sections)
  as vanishing loci of sections of rank $2$ ACM bundles on quartic $3$-folds in $\P^4$.
\end{remark}
\begin{lemma}
  For a half-canonically embedded curve $C$ in $\P^4$ not lying on a quadric hypersurface,
  the multiplication map $\mu_{[C,\eta,L]}$ is bijective.
\end{lemma}
\begin{proof}
  Set $\vartheta = \mathcal{O}_C(1)$.  Of course $\mathcal{O}_C(2) = \omega_C$.
  The fact that $C$ does not lie on a quadric hypersurface is equivalent to the bijectivity
  of the multiplication map
  \begin{equation*}
    \mu_\vartheta\colon \Sym^2 H^0(C, \vartheta) \rightarrow H^0(C, \omega_C).
  \end{equation*}
  We now choose any closed point $x\in C$.  Using that $\vartheta$ is very ample we get
  \begin{equation*}
    h^0(C, \vartheta(-2x)) = h^0(C, \vartheta) - 2.
  \end{equation*}
  By Serre duality this implies
  $h^0(C, \vartheta(2x)) = h^0(C, \vartheta)$.  Let $L = \vartheta(2x)$, so $L$
  is a complete $\mathfrak{g}^4_{16}$ and $2x$ is contained in the base locus of $L$.  In particular, we have
  $H^0(C, L) \cong H^0(C,\vartheta)$ and $|L| = |\vartheta| + 2x$.  Taking symmetric powers, we get
  \begin{equation*}
    \Sym^2 H^0(C,L) \cong \Sym^2 H^0(C, \vartheta) \cong H^0(C, \omega_C).
  \end{equation*}
  The space $H^0(C, L^{\otimes 2})$ is $18$-dimensional,
  and it decomposes via the natural inclusion $H^0(C, \vartheta^{\otimes 2}) \hookrightarrow H^0(C, L^{\otimes 2})$ as
  \begin{equation*}
    H^0(C, L^{\otimes 2}) \cong H^0(C, \omega_C) \oplus V \cong \Sym^2 H^0(C, L) \oplus V,
  \end{equation*}
  where $\dim V = 3$.  The sections in $\Sym^2 H^0(C, L)$ vanish to orders at least $4$ at $x$.
  By Riemann--Roch, the space $H^0(C, L^{\otimes 2})$ does contain sections vanishing to orders
  $0$, $1$ and $2$ at $x$.
\begin{dummyversion}
    (Since the base loci of $L^{\otimes 2}$, $L^{\otimes 2}(-x)$ and $L^{\otimes 2}(-2x)$
    are empty).
\end{dummyversion}
  By the previous analysis, they must span $V$.

  Choose a two-torsion bundle $\eta\in \Pic^0(C)[2]$ such that $H^0(C, \vartheta \otimes \eta) = 0$.
  Since $\Pic^0(C)[2]$ acts transitively on the theta-characteristics, such an $\eta$ always exists
  by a result of Mumford \autocite{Mum66}.
  Then we have
  \begin{equation*}
    h^0(C, L \otimes \eta) = h^0(C, \vartheta(2x) \otimes \eta) \leq h^0(C, \vartheta \otimes \eta) + 2 = 2.
  \end{equation*}
  By Riemann--Roch we must in fact have $h^0(C, L\otimes \eta) = 2$.  By construction,
  \[H^0(C, (L\otimes \eta)(-2x)) = H^0(C, \vartheta\otimes\eta) = 0,\]
  so the two sections of $L\otimes \eta$ vanish to orders $0$ and $1$ at $x$.
  We conclude that the map
  \begin{equation*}
    \Sym^2 H^0(C, L \otimes \eta) \rightarrow H^0(C, L^{\otimes 2})
  \end{equation*}
  is injective and its image is precisely $V$.
\end{proof}

\subsection{Irreducibility of some spaces of linear series}
\label{subsec:grd-irreducible}

We now want to prove the irreducibility of $\UnivGrdTors{4}{16}$, i.e., the moduli space
of triples $[C,\eta,L]$ where $[C,\eta]\in\Rg{15}$ and $L\in W^4_{16}(C)$.
This will show that for the
general triple $[C, \eta, L]$, the map $\mu_{[C,\eta,L]}$ is an isomorphism.  Notice that the pair $[C,L]$
constructed in Section \ref{subsec:curve-example}
is \emph{not} Petri general, so we need more than the existence of a unique component
of $\UnivGrdTors{4}{16}$ dominating $\Mg{15}$.
Nonetheless, this fact is what we are going to establish first in greater generality:
\begin{proposition}
  \label{prop:unique-irreducible-component}
  Let $g \geq 3$ and $\rho(g,r,d) = 0$.  Then there is a unique irreducible component
  of $\mathfrak{G}^{r,(2)}_d$ dominating $\Mg{g}$,
  i.e., containing the Petri general triple $[C,\eta,L]$.
\end{proposition}
\begin{proof}
  If $r = g - 1$, the only $\mathfrak{g}^r_d$ on a curve is the canonical bundle,
  so $\UnivGrdTors{r}{d} \cong \Rg{g}$ is irreducible.  Otherwise, set $k = g - d + r + 1 \geq 3$.
  We recall that the locus of Petri general pairs $[C,L]$ is a connected smooth open subset $U$
  of one irreducible component of $\UnivGrd{r}{d}$ \autocite{EH87}. 
  The restriction of $\UnivGrdTors{r}{d}$ to the preimage $U^{(2)}$ of $U$
  is smooth, so in order to show $U^{(2)}$ is irreducible we only have to show it is connected.

  Take a general $k$-gonal curve $[D,A]$.
  We then have $h^0(D, A^{\otimes l}) = l + 1$ for all $l \leq r+1$ (see \autocite{Ba89}).
  So there is a rational map
  \begin{center}
    \begin{tikzpicture}
      \node (A) { $\Psi\colon \UnivGrdTors{1}{k}$ };
      \node[right=0.6cm of A] (B) { $\UnivGrdTors{r}{d}$ };
      \draw[dashed,->] (A) -- (B);
    \end{tikzpicture}
  \end{center}
  defined by $[D,\eta,A] \mapsto [D,\eta,A^{\otimes r}]$.
  We claim $A^{\otimes r}$ is Petri general, i.e., the map
  \begin{equation*}
    \mu_{A^{\otimes r}} \colon H^0(D, A^{\otimes r}) \otimes H^0(D, \omega_D \otimes A^{\otimes (-r)})
    \rightarrow H^0(D, \omega_D)
  \end{equation*}
  is injective. 
  The aforementioned result of Ballico implies
  \begin{equation*}
    h^0(D, \omega_D \otimes A^{\otimes (-j)}) = (k-1)(r+1-j)
  \end{equation*}
  for all $j\leq r+1$.
  Note also that $g = (k-1)(r+1)$.
  By counting dimensions we find that $\mu_{A^{\otimes r}}$ is injective if and only if it is surjective.
  
  We write down the beginning of the long exact sequence coming from the base point free pencil trick:
  \begin{equation*}
    0 \rightarrow H^0(\omega_D \otimes A^{\otimes (-j-1)})
    \rightarrow H^0(A) \otimes H^0(\omega_D \otimes A^{\otimes (-j)})
    \rightarrow H^0(\omega_D \otimes A^{\otimes (-j+1)}).
  \end{equation*}
  Comparing dimensions we find that the map on the right is surjective for all $j \leq r$.
  Now note that $h^0(D, A^{\otimes r}) = r+1$ is equivalent to $H^0(D, A^{\otimes r}) \cong \Sym^r H^0(D, A)$.
  The chain of surjective maps
  \begin{align*}
    H^0(A)^{\otimes r} \otimes H^0(\omega_D \otimes A^{\otimes(-r)})
    & \twoheadrightarrow
      H^0(A)^{\otimes(r-1)} \otimes H^0(\omega_D \otimes A^{\otimes(-r+1)}) \twoheadrightarrow \dotsb\\
    \dotsb & \twoheadrightarrow H^0(A) \otimes H^0(\omega_D \otimes A^{-1})
  \end{align*}
  then implies that the Petri map
  \begin{equation*}
    \mu_{A^{\otimes r}} \colon \Sym^r H^0(D, A) \otimes H^0(D, \omega_D \otimes A^{\otimes (-r)})
    \rightarrow H^0(D, \omega_D)
  \end{equation*}
  is surjective as well.  So $[D,\eta,A^{\otimes r}]$ lies in $U^{(2)}$.
  
  In \autocite{BF86} it is shown that the Hurwitz space $\UnivGrdTors{1}{k}$ is irreducible for $k\geq 3$.
  Hence $\Psi$ maps to the smooth locus of a unique component $Z$ of $\UnivGrdTors{r}{d}$ and its
  image is an irreducible subset consisting generically of Petri general curves.  Since the image is closed
  under monodromy of $2$-torsion, it follows that $U^{(2)}$ must be connected.
\end{proof}
We employ this result to prove irreducibility of $\UnivGrdTors{r}{d}$ under
special circumstances:
\begin{corollary}
  Assume $g\geq 3$ and $\rho(g,r,d) = 0$.  If $r \leq 2$ 
  or $r' = g - d + r - 1 \leq 2$, then $\UnivGrdTors{r}{d}$ is irreducible.
\end{corollary}
\begin{proof}
  Note that the Serre dual of a $\mathfrak{g}^r_d$ is a $\mathfrak{g}^{r'}_{2g - 2 - d}$,
  so the space $\UnivGrdTors{r}{d}$ is irreducible if and only if $\UnivGrdTors{r'}{2g-2-d}$ is.
  As mentioned above, if $r = 0$ or, equivalently, $r' = g-1$, the unique $\mathfrak{g}^r_d$
  on a curve is its canonical bundle, so $\UnivGrdTors{r}{d} \cong \Rg{g}$ is irreducible.
  The case $r = 1$ is just the aforementioned result by Biggers and Fried \autocite{BF86}
  about the irreducibility of Hurwitz spaces.

  In the remaining case $r = 2$ a general $\mathfrak{g}^2_d$
  maps $C$ birationally to a nodal curve in $\P^2$.  Thus we get a dominant rational map
  \begin{center}
    \begin{tikzpicture}
      \node (A) { $\strut V^{d,g}$ };
      \node[right=0.6cm of A] (B) { $\strut\UnivGrd{2}{d}$ };
      \draw[dashed,->] (A) -- (B);
    \end{tikzpicture}
  \end{center}
  from the Severi variety $V^{d,g}$ of irreducible plane curves of degree $d$ and arithmetic genus $g$.
  The Severi varieties are irreducible, as proven in \autocite{Har84}, so $\UnivGrd{2}{d}$
  is irreducible as well.

  Étale maps preserve dimension, so all components of $\UnivGrdTors{2}{d}$
  have dimension $3g - 3 + \rho(g,r,d) = 3g - 3$.  Each component is generically smooth,
  which implies that the general element has injective Petri map.  But by
  Proposition \ref{prop:unique-irreducible-component} there is only one such component.
\end{proof}
\begin{dummyversion}
  In order to see that the general $\mathfrak{g}^r_d$ maps $C$ birationally
  to a nodal curve in $\P^2$ we have to carry out a dimension count for
  components where the general element is composed with an involution.
  For instance, the dimension of components contributed by hyperelliptic
  curves does not exceed $3g - 4$ and that says it all.
\end{dummyversion}
In particular, $\UnivGrdTors{4}{16}$ is irreducible.
We may therefore specialize a general triple $[C, \eta, L] \in \UnivGrdTors{4}{16}$ to the previously
constructed explicit example.
This proves that the locus $\MyDivisorBar$ is a genuine divisor.  We proceed to calculate its class.

\subsection{Calculation of the divisor class}

Recall that we are considering the morphism
\begin{equation*}
  \degenmorph\colon \Sym^2 \UnivCrdTorsMorph_\ast (\Poincarebundle) \oplus
  \Sym^2 \UnivCrdTorsMorph_{\ast} (\Poincarebundle \otimes \PrymPoincare)
  \rightarrow \UnivCrdTorsMorph_\ast (\Poincarebundle^{\otimes 2})
\end{equation*}
between vector bundles of the same rank.  To calculate the Chern classes of these bundles
we will employ Grothendieck--Riemann--Roch.  For this we study the contribution
coming from $R^1 \UnivCrdTorsMorph_{\ast} (\Poincarebundle \otimes \PrymPoincare)$.
\begin{lemma}
  Let $[C,\eta] \in \Delta_0''$ be general and $L \in W^4_{16}(C)$.
  Then $h^0(C, L \otimes \eta) = 4$.
\end{lemma}
\begin{proof}
  Let $\nu\colon \normaliz{C} \rightarrow C$ be the normalization of $C$ and $x$ be the node.
  Then $\nu^\ast \eta = \mathcal{O}_{\normaliz{C}}$ and $\nu^\ast L \in W^4_{16}(\normaliz{C})$,
  since $\normaliz{C}$ is Brill--Noether general.
  From the exact sequence
  \begin{equation*}
    0 \rightarrow \mathcal{O}_C \rightarrow \nu_\ast \mathcal{O}_{\widetilde{C}} \xrightarrow{e} \CC_x \rightarrow 0
  \end{equation*}
  we get
  \begin{equation*}
    0 \rightarrow L\otimes \eta \rightarrow \nu_\ast \nu^\ast L \xrightarrow{e'} L\otimes \eta|_x \rightarrow 0,
  \end{equation*}
  and by the long exact sequence in cohomology we obtain
  \begin{equation*}
    0 \rightarrow H^0(C, L\otimes \eta) \rightarrow H^0(\normaliz{C}, \nu^\ast L) \xrightarrow{H^0(e')} \CC.
  \end{equation*}
  Now $H^0(e)$ is the zero map, hence $H^0(e')$ must be nonzero and we get
  \begin{equation*}
    h^0(C, L\otimes \eta) = h^0(\normaliz{C}, \nu^\ast L) - 1 = 4.\qedhere
  \end{equation*}
\end{proof}
This implies that the dimension of $h^0(C, L \otimes \eta)$ jumps by two on the
boundary component $\Delta_0''$.  Hence $R^1 \UnivCrdTorsMorph_\ast ( \Poincarebundle \otimes \PrymPoincare )$
is supported at least on $\Delta_0''$, and there it is of rank $2$.
\begin{remark}
  In fact, $\Delta_0''$ seems to be the only divisor where
  $R^1 \UnivCrdTorsMorph_\ast( \Poincarebundle \otimes \PrymPoincare )$
  is supported.  Since a proof of this would take long, and is not strictly necessary
  to achieve the goal of the article, we do not assume this fact here and will discuss
  it in greater generality in future work.
\end{remark}
Denote $\mathfrak{d} = c_1 ( R^1 \UnivCrdTorsMorph_\ast ( \Poincarebundle \otimes \PrymPoincare ) )$.
\begin{proposition}
  The pushforward to $\RgSuit{15}$ of the class of the degeneracy locus
  $Z_1(\DegenMorph)$ is
  \begin{equation*}
    [\MyDivisorBar]^{\mathrm{virt}} \equiv
    31020 \left( \frac{3127}{470} \lambda - (\delta_0' + \delta_0'') - \frac{3487}{1880} \dram \right)
    - 3 \UnivGrdTorsMorph_\ast( \mathfrak{d} )
  \end{equation*}
  and $[\MyDivisorBar]^{\mathrm{virt}} - n[\MyDivisorBar]$ is an effective class
  supported on the boundary for some $n \geq 1$.
\end{proposition}
\begin{proof}
  We introduce the following classes in $A^1(\UnivGrdTorsPCZero{4}{16})$:
  \begin{equation*}
    \mathfrak{a} = \UnivCrdTorsMorph_\ast(c_1^2(\Poincarebundle)),\quad
    \mathfrak{b} = \UnivCrdTorsMorph_\ast(c_1(\Poincarebundle)\cdot c_1(\omega_\UnivCrdTorsMorph)),\quad
    \mathfrak{c} = c_1(\UnivCrdTorsMorph_\ast(\Poincarebundle)).
  \end{equation*}
  By Porteous' formula, the class of the first degeneracy locus $Z_1(\degenmorph)$ of $\degenmorph$ is given by
  \begin{equation*}
    Z_1(\degenmorph) = c_1( \concreteF )
    - c_1( \Sym^2 \UnivCrdTorsMorph_\ast \Poincarebundle )
    - c_1( \Sym^2 \UnivCrdTorsMorph_\ast (\Poincarebundle \otimes \PrymPoincare) ).
  \end{equation*}
  For a vector bundle $\mathcal{G}$ we have the elementary fact
  \begin{equation*}
    c_1( \Sym^2 \mathcal{G} ) = (\rk( \mathcal{G} ) + 1 ) c_1(\mathcal{G}).
  \end{equation*}
  Furthermore, for every $[C, \eta] \in \RgSuit{g}$ and every $L \in W^4_{16}(C)$ we have $H^1(C, L^{\otimes 2}) = 0$,
  so $R^1 \UnivCrdTorsMorph_\ast ( \Poincarebundle^{\otimes 2} ) = 0$.
  We can then apply Grothendieck--Riemann--Roch and express everything
  in terms of the classes $\mathfrak{a}$, $\mathfrak{b}$, $\mathfrak{c}$ and $\mathfrak{d}$.
  For instance we have
  \begin{align*}
    c_1( \concreteF )
    & = \Big[ \UnivCrdTorsMorph_\ast 
      \left(1 + c_1( \Poincarebundle^{\otimes 2} ) + \tfrac{1}{2}c_1^2 (\Poincarebundle^{\otimes 2})\right)\\
    &
      ~~~~~~~~~~~~\cdot \left(1 - \tfrac{1}{2}c_1(\omega_\UnivCrdTorsMorph)
      + \tfrac{1}{12}(c_1^2(\omega_\UnivCrdTorsMorph) + c_2(\Omega_\UnivCrdTorsMorph))\right)
      \Big]_1\\
    & = \lambda + 2\mathfrak{a} - \mathfrak{b},
  \end{align*}
  where $[-]_1$ denotes the degree $1$ part of an expression.  We have used Mumford's formula
  to calculate
  $\UnivCrdTorsMorph_\ast (c_1^2(\omega_\UnivCrdTorsMorph) + c_2(\Omega_\UnivCrdTorsMorph)) = 12 \lambda$.
  Similarly, also using Lemma \ref{lem:pushforward-dram}, we find
  \begin{equation*}
    c_1( \UnivCrdTorsMorph_\ast ( \Poincarebundle \otimes \mathscr{P} ) )
    = \lambda + \tfrac{1}{2}\mathfrak{a} - \tfrac{1}{2}\mathfrak{b} - \tfrac{1}{4}\dram + \mathfrak{d}.
  \end{equation*}
  Using the results of \autocite{F09}, in particular Lemmata 2.6 and 2.13 as well as Proposition 2.12,
  we can calculate the pushforwards of $\mathfrak{a}$, $\mathfrak{b}$ and $\mathfrak{c}$ by $\UnivGrdTorsMorph$:
  \begin{align*}
    \UnivGrdTorsMorph_\ast (\mathfrak{a}) & = -146784 \lambda + 20856 \dprimeboundary + 41712 \dram,\\
    \UnivGrdTorsMorph_\ast (\mathfrak{b}) & = 4224 + 264 \dprimeboundary + 528 \dram,\\
    \UnivGrdTorsMorph_\ast (\mathfrak{c}) & = -48279 + 6930 \dprimeboundary + 13860 \dram,
  \end{align*}
  and of course $\UnivGrdTorsMorph_\ast( \lambda ) = N\lambda$, $\UnivGrdTorsMorph_\ast(\dram) = N\dram$,
  where $N = 6006$ is the degree of $\UnivGrdTorsMorph$.
  Putting everything together, we obtain the result.
  The difference between $[\MyDivisorBar]^{\mathrm{virt}}$ and $[\MyDivisorBar]$ arises
  from the boundary components where $\phi$ is degenerate.  
\end{proof}
\begin{dummyversion}
  \begin{remark}
    A more detailed calculation:
    \begin{align*}
      Z_1(\degenmorph)
      & = \UnivGrdTorsMorph_\ast ( \lambda  + 2 \mathfrak{a} - \mathfrak{b} ) - 6 \UnivGrdTorsMorph_\ast (\mathfrak{c})
        - 3 \UnivGrdTorsMorph_\ast \left( \lambda + \frac{\mathfrak{a}}{2}
        - \frac{\mathfrak{b}}{2} - \frac{\dram}{4} + \mathfrak{d}\right)\\
      & = -2N \lambda + \frac{3}{4}N \dram
        + \UnivGrdTorsMorph_\ast \left( \frac{\mathfrak{a}}{2} + \frac{\mathfrak{b}}{2}
        - 6 \mathfrak{c} - 3 \mathfrak{d}\right)\\
      & = 206382 \lambda - 31020 (\delta_0' + \alpha \delta_0'') - \frac{115071}{2} \dram\qedhere
    \end{align*}
  \end{remark}
\end{dummyversion}
\begin{theorem}
  $\RgBar{15}$ is of general type.
\end{theorem}
\begin{proof}
  The contribution of $\sigma_\ast (\mathfrak{d})$ to $[\MyDivisorBar]$
  only improves the ratio between the coefficients of $\lambda$ and the boundary components.
  The same goes for the boundary components where $\phi$ is degenerate.
  Hence we may as well work with the class
  $[\MyDivisorBar]^{\mathrm{virt}} + 3\UnivGrdTorsMorph_\ast (\mathfrak{d})$.
  Then we take an appropriate linear combination of $\MyDivisorBar$
  and the divisor $\overline{\mathcal{D}}_{15:2}$
  from \cite{FL} having class
  \begin{align*}
    [\overline{\mathcal{D}}_{15:2}] & = 5808 \lambda - 924 \dprimeboundary - 990 \dram \\
    & = 924 \left( \tfrac{44}{7} \lambda - \dprimeboundary - \tfrac{15}{14} \dram  \right).
  \end{align*}
  For instance we have
  \begin{equation*}
    \beta \overline{\mathcal{D}}_{15:2} + \gamma \MyDivisorBar = \epsilon \lambda - 2 \dprimeboundary - 3 \dram,
  \end{equation*}
  where
  \begin{equation*}
    \beta = \frac{667}{680394}, \quad \gamma = \frac{4}{113399}, \quad \epsilon = \frac{10288}{793}.
  \end{equation*}
  Here $\epsilon < 13$, hence the canonical class is big.
\end{proof}
\begin{remark}
  The map
  \begin{equation*}
    \Sym^2 H^0(C, L \otimes \eta) \rightarrow H^0(C, L^{\otimes 2})/\Sym^2 H^0(C, L)
  \end{equation*}
  is identically zero along the boundary component $\Delta_0''$.  Hence the morphism
  $\degenmorph$ is degenerate with order $3$ along $\Delta_0''$.  It follows that we can
  subtract $3 \delta_0''$ from $Z_1(\degenmorph)$ and still obtain an effective class.
\end{remark}
\begin{dummyversion}
  \begin{proof}
    As before, we have the exact sequences
    \begin{equation*}
      0 \rightarrow L \rightarrow \nu^\ast \nu_\ast L \rightarrow L|_x \rightarrow 0
    \end{equation*}
    and
    \begin{equation*}
      0 \rightarrow L\otimes \eta \rightarrow \nu^\ast \nu_\ast L \rightarrow L\otimes \eta|_x \rightarrow 0
    \end{equation*}
    Hence we have the identification $H^0(C, L) \cong H^0(\normaliz{C}, \nu^\ast L)$
    and the inclusion $H^0(C, L \otimes \eta) \hookrightarrow H^0(\normaliz{C}, \nu^\ast L)$.
    It follows that we have an inclusion $|L \otimes \eta| \hookrightarrow |L|$ on the level of Weil divisors,
    which implies the statement.
  \end{proof}
\end{dummyversion}

\printbibliography

\Addresses

\end{document}